\documentclass[11pt,a4j, english,final]{article}
\usepackage{amsmath, amssymb}
\usepackage{amsfonts}
\usepackage{amscd}
\usepackage{amsthm}
\usepackage{color}
\usepackage{cases}
\usepackage[dvipdfmx]{graphicx}
\usepackage{tikz}
\usepackage{overpic}
\usepackage{here}
\usepackage{indentfirst}
\usepackage{showkeys}
\numberwithin{equation}{section}
\newtheorem{thm}{Theorem}[section]

\newtheorem{lem}[thm]{Lemma}
\newtheorem{cor}[thm]{Corollary}

\newtheorem{example}[thm]{Example}
\newtheorem{caution}[thm]{Remark}

 \newcommand{\N}{\mathbb{N}}
 \newcommand{\Z}{\mathbb{Z}}
 \newcommand{\Q}{\mathbb{Q}}
 \newcommand{\R}{\mathbb{R}}

 \newcommand{\epsi}{\varepsilon}
 \newcommand{\neron}[1]{\mathrm{NS}({#1})}

 \newcommand{\nef}{\mathrm{Nef}(\mathrm{Km}^2(A))}
 \newcommand{\mov}{\mathrm{Mov}(\mathrm{Km}^2(A))}
 \newcommand{\rk}{\mathrm{rk}}

 \newcommand{\innerpro}[2]{\langle {#1}, {#2} \rangle}

 \newcommand{\Hilb}{\mathrm{Hilb}}
 \newcommand{\Km}{\mathrm{Km}}
 \newcommand{\NS}{\mathrm{NS}}
 \newcommand{\Nef}{\mathrm{Nef}}
 \newcommand{\Mov}{\mathrm{Mov}}
 \newcommand{\Endo}{\mathrm{End}}
 \newcommand{\Amp}{\mathrm{Amp}}
\begin{document}

\begin{center}
\Large{{\bf Nef Cone of a Generalized Kummer 4-fold}}
\end{center}

\begin{center}
\normalsize{Akira Mori} 
\end{center}

\setcounter{section}{-1}

In this note, we calculate the boundary of movable cones and nef cones of the generalized Kummer 4-fold $\mathrm{Km}^2(A)$ attached to an abelian surface $A$ with $\rk \neron{A} = 1$.

\section{Introduction}

For an abelian surface $A$, Beauville \cite{Be} constructed a series of irreducible symplectic manifolds
$\Km^{l-1}(A)$ of dimension $2(l-1)$ $(l \geq 2)$.
Let $S^l(A)$ be the $l$-th symmetric product of $A$ and
$\Hilb^l(A)$ the Hilbert scheme of $l$ points on $A$.
Thus $\Hilb^l(A)$ parameterizes 0-dimension subscheme $Z$ of $A$
of $\chi({\cal O}_Z)=l$.
We have Hilbert-Chow morphism $\varphi:\Hilb^{l}(A) \to S^{l}(A)$
sending a subscheme $Z$ of length $l$ to the 0-cycle $[Z]$ defined by $Z$.  
Then $\Km^{l-1}(A)$ is the fiber of the morphism $\Hilb^{l}(A) \to S^{l}(A) \overset{\sigma}{\to} A$,
where $\sigma(x_1,x_2,...,x_{l})=\sum_i x_i$.
Thus we have the following commutative diagram:
\begin{equation}
\begin{CD}
\Km^{l-1}(A) @>>> \Hilb^{l}(A) @.\\
@V{\varphi}VV @VV{\varphi}V @.\\
\sigma^{-1}(0) @>>> S^{l}(A) @>{\sigma}>> A
\end{CD}
\end{equation}
If $l=2$, then 
$\Km^1(A)$ is nothing but the Kummer K3 surface of $A$,
Hence $\Km^{l-1}(A)$ $(l \geq 3)$ is called a generalized Kummer manifold.

For an irreducible symplectic manifold $M$, $H^2(M,{\Bbb Z})$ and hence Neron-Severi group
$\NS(M)$ has a bilinear form called the Beauville-Bogomolov form \cite{Be}. 
Assume that $A$ is an abelian surface with the Picard rank $\rho(A)=1$.
Let $H$ be an ample generator of $\NS(A)$ and set $n:=H^2/2$.
Then
$\NS(\Km^{l-1}(A))$ ($l \geq 3$) is described as
\begin{equation}
\NS(\Km^{l-1}(A)) \cong {\Bbb Z}h \oplus \Z \delta
\end{equation}
and the Beauville-Bogomolov bilinear form satisfies
$$
h^2=2n,\; \delta^2=-2l,\;(h,\delta)=0,
$$ 
where $h$ is the pull-back of an ample divisor on $S^{l}(A)$
by $\Km^{l-1}(A) \to S^{l}(A)$ (cf. \cite[Proposition 4.11]{Yoshi1}). 
Let $D$ be the exceptional divisor
of $\Km^{l-1}(A) \to  \sigma^{-1}(0)$.
Then $D \in |2\delta|$ and 
$kh-\delta$ is ample for $k \gg 0$.

In \cite{Yoshi1}, Yoshioka gave a lattice theoretic description of
the movable cone $\Mov(\Km^{l-1}(A))$ and the nef cone $\Nef(\Km^{l-1}(A))$ of $\Km^{l-1}(A)$
($l \geq 3$).
In this note, we shall give a more concrete description of
$\Nef(\Km^2(A))$ for an abelian surface of $\rho(A)=1$.


In order to state our main result,
let us consider a Diophantine equation 
\begin{equation}\label{eq:Pell}
3Y^2-nX^2=3
\end{equation}
of Pell type.
If $\sqrt{n/3} \not \in \Q$, then
let $(X_1,Y_1)$ ($X_1,Y_1>0$) be the fundamental solution of \eqref{eq:Pell}, that is,
$(X_1,Y_1)$ is a solution of \eqref{eq:Pell}   
minimizing $X$ in the solution $(X,Y)$ of \eqref{eq:Pell} with $X,Y>0$. 
We define $(X_k,Y_k)$ ($k \geq 1$) by
$$
Y_k+\sqrt{n/3}X_k=
(Y_1+\sqrt{n/3}X_1)^k.
$$  
We also set $(X_0,Y_0):=(0,1)$.
Then $X_k,Y_k$ are positive integers satisfying \eqref{eq:Pell} and
$\{(\pm X_k,\pm Y_k) \mid k \geq 0\}$ is the set of all solutions.

\begin{thm}\label{thm:main}
The movable cone $\mov$ and the nef cone $\nef$ 
of $\Km^2(A)$ are characterized by the solution of Pell equation \eqref{eq:Pell}
and $n$ as following table: 
\begin{center}
{\renewcommand\arraystretch{1.5}
\begin{tabular}{|c|c|c|c|c|c|c|}

\hline
\multicolumn{2}{|c|}{type of $n$} & \multicolumn{3}{|c|}{type of $(X_1,Y_1)$} & $\nef$ & $\mov$ \\ \hline \hline

\multicolumn{2}{|c|}{$3 \nmid n$} & \multicolumn{3}{|c|}{$3 \mid X_1$} & $\R_{\geq 0}h + \R_{\geq 0}(h-\frac{nX_1}{3Y_1}\delta)$ & $\R_{\geq 0}h + \R_{\geq 0}(h-\frac{nX_1}{3Y_1}\delta)$ \\ \hline

 &  & \multicolumn{3}{|c|}{$3 \mid X_1$} & $\R_{\geq 0}h + \R_{\geq 0}(h-\frac{nX_1}{3Y_1}\delta)$ & $\R_{\geq 0}h + \R_{\geq 0}(h-\frac{nX_1}{3Y_1}\delta)$ \\ \cline{3-7}

 &  &  & $X_1$ & $3 \mid Y_1$ & $\R_{\geq 0}h + \R_{\geq 0}(h-\frac{nX_1}{3Y_1}\delta)$ & $\R_{\geq 0}h + \R_{\geq 0}(h-\frac{nX_2}{3Y_2}\delta)$ \\ \cline{5-7} 

$n=3m$ & $m$ is not & $3 \nmid X_1$ & even & $3 \nmid Y_1$ & $\R_{\geq 0}h + \R_{\geq 0}(h-\frac{nX_1}{3Y_1}\delta)$ & $\R_{\geq 0}h + \R_{\geq 0}(h-\frac{nX_3}{3Y_3}\delta)$ \\ \cline{4-7}

 & square &  & $X_1$ & $3 \mid Y_1$ & $\R_{\geq 0}h + \R_{\geq 0}(h-\frac{nX_2}{3Y_2}\delta)$ & $\R_{\geq 0}h + \R_{\geq 0}(h-\frac{nX_2}{3Y_2}\delta)$ \\ \cline{5-7}

 &  &  & odd & $3 \nmid Y_1$ & $\R_{\geq 0}h + \R_{\geq 0}(h-\frac{nX_2}{3Y_2}\delta)$ & $\R_{\geq 0}h + \R_{\geq 0}(h-\frac{nX_3}{3Y_3}\delta)$ \\ \cline{2-7}

 & $m$ is square & \multicolumn{3}{|c|}{} & 
$\R_{\geq 0} h+\R_{\geq 0}(h-\sqrt{\frac{n}{3}}\delta)$ & 
$\R_{\geq 0} h+\R_{\geq 0}(h-\sqrt{\frac{n}{3}}\delta)$ \\ \hline

\end{tabular}
}
\end{center}

\end{thm}

As we already know the trivial boundary of $\Nef(\Km^2(A))$ defining Hilbert-Chow  contraction,
we shall describe the other boundary.
We show that it is defined
by an isotropic vector.  
Then we have Pell type equation \eqref{eq:Pell}
and get Theorem \ref{thm:main}. 
By the proof of Theorem \ref{thm:main},
we also give a chamber decomposition of $\Mov(\Km^2(A))$ in section \ref{sect:chamber}.
In section \ref{sect:appendix},
we give $\Mov(\Km^{l-1}(A))$ (Theorem \ref{thm:general}). 

{\it{Acknowledgement}}. The author would like to thank my adviser, K$\mathrm{\bar{o}}$ta Yoshioka, for his support and advice. 

\section{Preliminary}

Let $\innerpro{\ }{\ }$ be the Mukai pairing on the algebraic cohomology groups $H^{\ast}(A, \Z)_{\mathrm{alg}} := \Z \oplus \NS(A) \oplus \Z$,. 
For $x = (x_0, x_1, x_2),x'=(x_0',x_1',x_2')$,
$\langle x,x' \rangle=x_1 x_1'-x_0 x_2'-x_2 x_0'$. 
We write $x^2 := \innerpro{x}{x}$. 
For $E \in \mathrm{Coh}(A)$, $v(E) = \mathrm{ch}(E) \in H^{\ast}(A, \Z)_{\mathrm{alg}}$ is the Mukai vector of $E$. 
Mukai vector $v = (r, \xi, a)$ is called primitive if $\gcd(r, \xi, a) = 1$.

Let $A$ be an abelian surface with $\NS(A) = \Z H$ where $H$ is a general ample divisor on $A$
 and $H^2 = 2n, n \in \N$. 
We set $v = (1, 0, -l)$. 
It is easy see that
$$
v^\perp=\Z(0,H,0) \oplus \Z(1,0,l)
$$
and we get an isometry  
\begin{equation}
\begin{matrix}
\theta_v: & v^\perp & \to & \Z h \oplus \Z \delta\\
& (0,H,0) & \mapsto & h\\
& (1,0,l) & \mapsto & \delta
\end{matrix}
\end{equation}
By $\theta_v$, we shall identify $v^\perp$ with $\NS(\Km^{l-1}(A))$.
We set
$$
P^+:=\{x \in v^\perp \mid x^2>0, \langle x,h \rangle>0\}.
$$

\section{Movable Cone and Nef Cone}

We recall a description of $\Mov(\Km^{l-1}(A))$ and
$\Nef(\Km^{l-1}(A))$ in \cite{Yoshi1}. 
 We consider the set $\Gamma$ of Mukai vector $u$ satisfying  the inequality 
$$\innerpro{u}{v-u} > 0,\ u^2 \geq 0, \ \langle (v-u)^2 \rangle \geq 0, \ \innerpro{v}{u}^2 > v^2 u^2. $$
If $u \in \Gamma$, then $u^{\perp}$ is not empty (\cite[Proposition 1.3]{Yoshi1}). 
The connected component ${\cal C}$ of $P^+ \setminus \cup_{u \in \Gamma} u^\perp$ containing 
$h - \epsi \delta\ (0 < \varepsilon \ll 1)$
is the ample cone $\Amp(\Km^{l-1}(A))$ of $\Km^{l-1}(A)$ (\cite[Proposition 4.11]{Yoshi1}).
For
$$
\Gamma_m := \{ u \in \Gamma \mid u^2 = 0,\ \innerpro{u}{v} = 1\ or\ 2 \},
$$
let ${\cal C}'$ be the connected component of $P^+\setminus \cup_{u \in \Gamma_m} u^\perp$
containing ${\cal C}$.
Then 
\begin{equation}
\Nef(\Km^{l-1}(A))=\overline{\cal C},\;
\Mov(\Km^{l-1}(A))=\overline{\cal C'}.
\end{equation}



\section{The Calculation of Boundary of Cones}

In this section, we shall prove Theorem \ref{thm:main}. 
We keep the notation in section 2 unless otherwise stated.
We first prove that the boundaries of $\Nef(\Km^2(A))$ are defined by
an isotropic vector $u \in \Gamma$. 

\begin{lem}\label{lem:isotropic}
Let $v = (1, 0, -l)$ and $l \leq 4$. 
Then $u \in \Gamma$ satisfies one of the following conditions:
\begin{enumerate}
\item
$u^2=0$ and $0<\langle u,v \rangle \leq l$.
\item
$(v-u)^2=0$ and $0<\langle v-u,v \rangle \leq l$.
\end{enumerate}
\end{lem}

\begin{proof}
We set $w = v-u$. Since $w \in \Gamma$,
we may assume that $u^2 \leq w^2$.
We shall prove $u^2=0$ and $0<\langle u,v \rangle \leq l$.
Since $u^2 \leq w^2$, we have $\innerpro{u}{v} \leq l$ .
Since $\innerpro{u}{w} > 0$ and $u^2 \geq 0$, we get
$$0 \leq u^2 < \innerpro{u}{v} \leq l. $$ 

Next we prove $u^2 = 0$. 
Since $v^2 \geq 2(u^2 + \innerpro{u}{w})$ and $l \leq 4$, we have $4 \geq u^2+\innerpro{u}{w}$. 
$w$ lies in $\Gamma$, so we get $4 > u^2 + \sqrt{u^2w^2} > 2u^2$. 
Since $u^2$ is even, the statement follows. 
\end{proof}

From now on, we assume that $l=3$, that is,
$v=(1, 0, -3)$. 
We take $u \in \Gamma$ defining a wall.   
We may assume that $u^2 = 0$ and $\innerpro{u}{v} = 1, 2$ or $3$ by Lemma \ref{lem:isotropic}. 
If we assume $u-\lambda v \in v^{\perp}$ where $\lambda \in \R$, then we have 
$\lambda = \innerpro{u}{v}/6. $
Hence, $u$ is represented by 
\begin{equation}
u=\frac{\innerpro{u}{v}}{6}v +xh+y\delta,\ \rm{where}\  \it{x, y} \in \Q. 
\end{equation} 
Since $\frac{6}{\innerpro{u}{v}}u-v \in v^{\perp} \cap H^{\ast}(X, \Z)_{\mathrm{alg}}$, we get 
\begin{equation}\label{eq:u}
\frac{6}{\innerpro{u}{v}}u = v + Xh + Y\delta,\ \rm{where}\  X, Y \in \Z.
\end{equation}
Since $u$ is isotropic, $(X,Y)$ is a solution of \eqref{eq:Pell}
satisfying 
\begin{equation}\label{eq:Pell2}
\gcd(Y+1,X,6)=2,3,6.
\end{equation} 
Conversely for an integral solution of \eqref{eq:Pell} satisfying \eqref{eq:Pell2},
we have a primitive isotropic Mukai vector $u$
satisfying \eqref{eq:u}.
For the isotropic vector $u$, the wall $u^\perp$ in $P^+$
is $\R_{>0}(h-\frac{nX}{3Y}\delta)$.


The following lemma shows that the slope converges monotonically when Pell equation has infinitely many solutions. 

\begin{lem}\label{lem:slope}
Assume that
$\sqrt{n/3} \not \in \Q$.
Then  
$$
0=\frac{X_0}{Y_0}<\frac{X_1}{Y_1}<\frac{X_2}{Y_2}<\frac{X_3}{Y_3}<\cdots,\;
\lim_{k \to \infty}\frac{X_k}{Y_k}=\sqrt{\frac{3}{n}}.
$$
\end{lem}




{\it Proof of Theorem \ref{thm:main}.}
We divide the proof into two cases. \quad \\
(1) We assume that $3 \nmid n$.
In this case, $\sqrt{3n} \not \in \Q$. 
Then  since $3 \mid X$, we get $(X_1,Y_1)=(3Z_1,Y_1)$, where
$(Z_1,Y_1)$ is the fundamental solution of Pell equation
$Y^2-(3n)Z^2=1$. 
By $3 \mid X_1$, we get $Y_1 \equiv \pm 1 \mod 3$.
Hence $\gcd(Y+1,X,6)=3,6$ for $(X,Y)=(X_1,Y_1),(-X_1,-Y_1)$.
Moreover $\gcd(Y+1,X,6)=6$ if and only if $2 \mid X_1$.
Therefore 
$$
u=
\begin{cases}
(\tfrac{\pm Y_1+1}{3}, \pm \tfrac{X_1}{3}H, \pm Y_1-1), & \text{if } 2 \nmid X_1\\
 (\tfrac{\pm Y_1+1}{6}, \pm \tfrac{X_1}{6}H, \tfrac{\pm Y_1-1}{2}), & \text{if } 2 \mid X_1
\end{cases}
$$
Hence we get $\nef = \mov$.

(2) We consider the case of $n=3m$ where $m$ is not a square number. 
Then \eqref{eq:Pell} is equivalent to the Pell equation $Y^2-(n/3)X^2=1$. 
We divide this case into four cases. 

(i)
We assume that $3 \mid X_1$.
Then 
we can get the statement as with the case of $3 \nmid n$.  

(ii)
We assume that $2 \mid X_1$ and $3 \nmid X_1$. 
Since $Y_1^2 \equiv 1 \mod 2$,
$\gcd(1\pm Y_1,\pm X_1,6)=2$ and
$\innerpro{u}{v} = 3$ for 
$u = (\frac{\pm Y_1+1}{2}, \pm \frac{X_1}{2}H, \frac{3}{2}(\pm Y_1-1))$. 
In order to determine a movable cone, we must consider next solution $(X_2, Y_2) = (2X_1Y_1, Y_1^2+mX_1^2) $ by Lemma \ref{lem:slope}. 

\begin{itemize}
 \item[$\bullet$] 
Assume that $3 \mid Y_1$. 
Since $Y_2=Y_1^2+mX_1^2=2Y_1^2-1$,
$\gcd(Y_2+1,X_2,6)=6$ and $\langle v,u \rangle=1$ for
$u=(\frac{Y_1^2}{3},\frac{Y_1 X_1}{3}H,m X_1^2)$.
 \item[$\bullet$] 
If $3 \nmid Y_1$, then $\gcd (1 \pm Y_2, \pm X_2,6) = 2$. 
Thus, we consider next solution $(X_3, Y_3) = (3X_1Y_1^2+mX_1^3, Y_1(Y_1^2+3mX_1^2))$. 
Note that $3 | m$ since $3 \nmid X_1, 3 \nmid Y_1$ and $Y_1^2-mX_1^2 = 1$. 
Since either $Y_3+1$ or $Y_3-1$ is a multiple of $54$, we have $\innerpro{u}{v} = 1$ for 
$u = (\frac{\pm Y_3+1}{6}, \pm \frac{X_3}{6}H, \frac{\pm Y_3-1}{2})$.  
\end{itemize}

(iii)
We assume that $2 \nmid X_1$ and $3 \nmid X_1$. 
Since $\gcd(1\pm Y_1, \pm X_1,6) = 1$, we consider next solution $(X_2, Y_2)$. 
\begin{itemize}
\item[$\bullet$] 
If $3 \mid Y_1$, then we also see that 
$\gcd(Y_2+1,X_2,6)=6$ and
$u=(\frac{Y_1^2}{3},\frac{Y_1 X_1}{3}H,m X_1^2)$.
\item[$\bullet$] 
If $3 \nmid Y_1$, then 
$\gcd (Y_2+1, X_2,6) = 2$
and $\langle u,v \rangle=3$ for
$u=(Y_1^2,Y_1 X_1 H,3m X_1^2)$.
We consider next solution $(X_3, Y_3)$. 
Since $X_3=X_1(3Y_1^2+mX_1^2)=X_1(4Y_1^2-1)$,
$2 \nmid X_3$.
By $Y_1^2 \equiv 1 \mod 3$ and $3 \nmid X_1$,
$3 \mid m$. Hence
we also have $3 \mid X_3$.
Therefore $\gcd(Y+1,X,6)=3$ and
$u=(\frac{Y+1}{3}, \frac{X}{3}H, Y-1)$
for $(X,Y)=(X_3,Y_3)$ or $(X,Y)=(-X_3,-Y_3)$. 
\end{itemize}

(iv)
We assume that $m$ is square number. 
Then the statement is showed by \cite[Proposition 4.16]{Yoshi1} and the fact that 
\eqref{eq:Pell} has only trivial solutions $(0, \pm 1)$. 
\qed
\begin{caution}\rm{
We state the characterization of solutions of Pell equation $Y^2 -m X^2 = 1$ by $m$: 
$$3 | m\ \Longrightarrow \ 3 \nmid Y,\ \ \ m = 3k+1\ \Longrightarrow \ 3|X\ \rm{and}\ 3 \nmid Y, $$
$$m = 3k+2\ \Longrightarrow \ 3 | X\ \rm{and}\ 3 \nmid Y,\ \rm{or}\ 3 \nmid X\ \rm{and}\ 3 | Y .$$
Actually, they follow by solving the Pell equation in the residue field $\mathbb{F}_3$. 
In particular, we can see that $3 | X_1$ if $m = 3k+1$. 
}
\end{caution}

\begin{cor}
If $3 \nmid n$ or $n \equiv 3 \mod 9$, then
$\Nef(\Km^2(A))=\Mov(\Km^2(A))$.
\end{cor}
\begin{example}\rm{
Let $n=1$. 
Then $(X, Y)$ satisfies the Pell equation $Y^2 -3(X/3)^2=1$. 
The minimum solution of this equation is $(X_1, Y_1) = (3, 2)$. 
Since $Y_1+1=3$ is a multiple of 3, 
$$\frac{6}{\innerpro{u}{v}}u=v+2h+3\delta = (3, 3H, 3) = 3(1, H, 1), $$
and $\innerpro{u}{v}=2$. 
That is, the boundaries of $\mov$ and $\nef$ are determined and matched for $u=(1, H, 1)$. 
Moreover, the slope of $u^{\perp}$ monotonically converges to $\sqrt{3}/3$. 
We illustrate the movement that the walls monotonically converge to the boundary of positive cone (see Fig 1). 
}

\begin{figure}[H]
\begin{center}
\begin{overpic}[width=100mm]{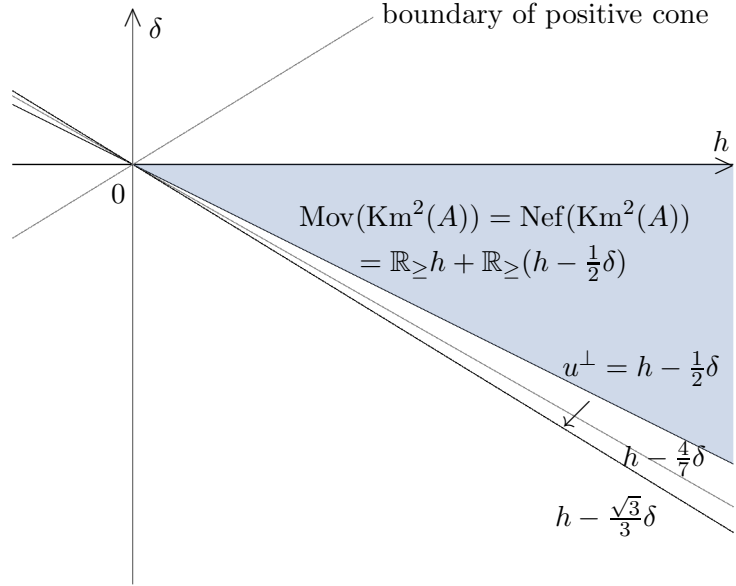}
 \put(20,75){$\delta$} 
 \put(15,53){$0$}   
 \put(51,77){boundary of positive cone} 
 \put(40,50){$\mov = \nef$}
 \put(48,44){$ = \R_{\geq}h + \R_{\geq}(h-\frac{1}{2}\delta)$}
 \put(95,60){$h$}
 \put(74,10){$h-\frac{\sqrt{3}}{3}\delta$} 
 \put(75,30){$u^{\perp} = h-\frac{1}{2}\delta$} 
 \put(83,18){$h-\frac{4}{7}\delta$} 
 \put(75,24){$\swarrow$} 
 \put(95,57){$>$}
 \put(16.6,76.5){$\wedge$}
\end{overpic}
\end{center}
\caption{{\footnotesize The walls monotonically converge to the boundary of positive cone}}
\end{figure}

\end{example}
By the proof of Theorem \ref{thm:main},
we have the following.
\begin{cor}\label{cor:H-C}
Assume that $\sqrt{3n} \not \in \Q$.
In the following cases, the boundaries of $\Mov(\Km^2(A))$ are given by 
Hilbert-Chow contractions: 
\begin{enumerate}
\item
$2 \mid X_1$.
\item
$3 \mid n$, $X_1 \equiv \pm 1 \mod 6$ and $3 \mid Y_1$.
\end{enumerate}
\end{cor}

\section{Chamber decomposition of $\Mov(\Km^2(A))$}\label{sect:chamber}

For $\Mov(\Km^2(A))$, we have a chamber decomposition
such that each chamber is an ample cone of a minimal model of $\Km^2(A)$.
We shall describe the decomposition.
By Theorem \ref{thm:main},
it is sufficient to treat the following 3 cases.
For the other cases, $\Nef(\Km^2(A))=\Mov(\Km^2(A))$.
\begin{enumerate}
\item[(1)]
$3 \nmid X_1$, $2 \mid X_1$ and $3 \mid Y_1$.
\item[(2)]
$3 \nmid X_1$, $2 \mid X_1$ and $3 \nmid Y_1$.
\item[(3)]
$3 \nmid X_1$, $2 \nmid X_1$ and $3 \nmid Y_1$.
\end{enumerate}

Case (1).
Since $Y_1^2-m X_1^2=1$,
$m \equiv -1 \mod 3$.
Then $\Mov(\Km^2(A))$ has two chambers 
\begin{equation}
\begin{split}
{\cal C}_1:=& \R_{>0} h+\R_{>0} (h-\tfrac{nX_1}{3Y_1}),\\
{\cal C}_2:=& \R_{>0} (h-\tfrac{nX_1}{3Y_1})+\R_{>0} (h-\tfrac{nX_2}{3Y_2}).
\end{split}
\end{equation}
Let $M_1:=\Km^2(A), M_2$ be the minimal models such that
$\Amp(M_i)={\cal C}_i$.
Then $\langle u,(1,0,-3) \rangle$ 
$=1$ for 
$u=(\frac{Y_1^2}{3},\frac{Y_1 X_1}{3}, mX_1^2)$
and
$M_2 \cong \Km^2(A')$, where $A':=M_H(u)$.
$M_2$ is a flop of $M_1$ along copies of ${\Bbb P}^2$.
If $\Endo(A) \cong \Z$, then
Lemma \ref{lem:abel} implies $A' \not \cong A$ and $M_1 \not \cong M_2$.

Case (2).
$3 \mid m$ and
$\Mov(\Km^2(A))$ is divided into 3 chambers 
\begin{equation}
\begin{split}
{\cal C}_1:=& \R_{>0} h+\R_{>0} (h-\tfrac{nX_1}{3Y_1}),\\
{\cal C}_2:=& \R_{>0} (h-\tfrac{nX_1}{3Y_1})+\R_{>0} (h-\tfrac{nX_2}{3Y_2}),\\
{\cal C}_2:=& \R_{>0} (h-\tfrac{nX_2}{3Y_2})+\R_{>0} (h-\tfrac{nX_3}{3Y_3}).
\end{split}
\end{equation}
Let $M_1=\Km^2(A), M_2, M_3$ be the minimal models with $\Amp(M_i)={\cal C}_i$
By Corollary \ref{cor:H-C},
$M_3 \cong \Km^2(A')$, where
$A'=M_H(u)$.  If  $\Endo(A) \cong \Z$ and
$u \ne (na^2, abH,b^2)$, then
Lemma \ref{lem:abel} implies $A' \not \cong A$ and $M_1 \not \cong M_3$.

Case (3).
$3 \mid m$ and
$\Mov(\Km^2(A))$ has two chambers 
\begin{equation}
\begin{split}
{\cal C}_1:=& \R_{>0} h+\R_{>0} (h-\tfrac{nX_2}{3Y_2}),\\
{\cal C}_2:=& \R_{>0} (h-\tfrac{nX_2}{3Y_2})+\R_{>0} (h-\tfrac{nX_3}{3Y_3}).
\end{split}
\end{equation}
Let $M_1:=\Km^2(A), M_2$ be the minimal models such that
$\Amp(M_i)={\cal C}_i$.
By the proof of Theorem \ref{thm:main} case (iii),
$\langle u,(1,0,-3) \rangle \ne \pm 1$.
Therefore $M_1 \not \cong M_2$.

\begin{lem}\label{lem:abel}
\begin{enumerate}
\item[(1)]
For a solution $(X,Y)$ of \eqref{eq:Pell},
$u=(na^2,abH,b^2)$ with $\langle u,(1,0,-3) \rangle =\pm 1$ if and only if
$Y \equiv -1 \mod p$ for all prime divisors $p>2$ of $n$  
and $Y \equiv -1 \mod 4$ if $n$ is even.
\item[(2)]
Assume that $\Endo(A) \cong \Z$.
Then $M_H(u) \cong A$ if and only if
$(s,t)=(n,1)$. 
\item[(3)]
We assume that
$\Endo(A) \cong \Z$ and
$\Nef(\Km^2(A)) \ne \Mov(\Km^2(A))$. Then  
$\Km^2(M_H(u))$ 
$\cong \Km^2(A)$ if and only if
$Y \equiv -1 \mod p$ for all prime divisors $p>2$ of $n$  
and $Y \equiv -1 \mod 4$ if $n$ is even.
\end{enumerate}
\end{lem}

\begin{proof}
(1)
We write $u=(s a^2,abH,tb^2)$ with $st=n$.
$3na^2-b^2=\pm 1$ implies $p \nmid (Y-1)$ for all prime divisors $p>2$ of $n$.
Moreover if $2 \mid n$, then $(Y-1)/2$ is odd.

Conversely if the conditions hold, then 
$\gcd(s,t)=1$. Hence $p \nmid (Y-1)$ for all prime divisors $p>2$ of $n$.
Moreover if $2 \mid n$, then $(Y-1)/2$ is odd.
Therefore $s=n$.

(2) The first claim is a consequence of \cite[Lemma 7.3]{YY1}.
(3) If $\Km^2(M_H(u)) \cong \Km^2(A)$, then 
the isomorphism preserves the Hilbert-Chow contractions.
In particular the isomorphism induces an isomorphism
of the exceptional divisors.
Since the Albanese varieties are $M_H(u)$ and $A$ respectively \cite{Nam1},
we have $M_H(u) \cong A$.
\end{proof}

\section{Movable Cones of a Generalized Kummer manifold}\label{sect:appendix}

As an appendix,
we calculate $\Mov(\Km^{l-1}(A))$ for an abelian surface $A$ with $\rho(A)=1$.
Let $u$ be the Mukai vector which determines the non-trivial boundary of $\Mov(\Km^{l-1}(A))$.
By using the same way as in the case of $l=3$, we have 
$$
\frac{2l}{\innerpro{u}{v}}u = v + Xh + Y\delta = (1+Y, XH, l(Y-1)),
$$
where $X, Y \in \Z$. 
Since $u$ is an isotropic vector, we have $lY^2 -n X^2 = l$. 
Moreover, since $u$ satisfies $\innerpro{u}{v} = 1, 2$, we have $l \mid X$.
Let $X = lZ$. 
Then we have
\begin{equation}\label{eq:Pell3}  
Y^2 -lnZ^2 = 1
\end{equation}
Let $(Z_1, Y_1)$ be the minimum solution of \eqref{eq:Pell3}. 
Then it satisfies 
$$
Y_1^2 -ln{Z_1}^2 = 1 \Longleftrightarrow (Y_1+1)(Y_1-1) = lnZ_1^2.
$$
\begin{itemize}
\item[(1)] 
If $l | (Y_1+1)$ or $l | (Y_1-1)$, let $Y_1 \pm 1 = kl$, where $k \in \N$. 
The vectors 
$$
(1\pm Y_1, \pm X_1H, l(\pm Y_1-1)) = (\pm kl, \pm lZ_1H, l(\pm kl - 2)) = \pm l(k, Z_1H, kl \mp 2)
$$
are divided by $l$. 
Then $\gcd (k, Z_1, kl \mp 2) = \gcd (k, Z_1, 2) = 1$ or $2$. 
If $\gcd (k, Z_1, 2) = 1$, 
we have $\innerpro{u}{v} = 2$ for $u = (\frac{1\pm Y_1}{l}, \pm \frac{X_1}{l}H, \pm Y_1-1)$.
If $\gcd (k, Z_1, 2) = 2$, 
we have $\innerpro{u}{v} = 1$ for $u = (\frac{1\pm Y_1}{2l}, \pm \frac{X_1}{2l}H, \frac{\pm Y_1-1}{2})$.  
 \item[(2)] 
If $l \nmid (Y_1 \pm 1)$, we consider the next solution $(Z_2, Y_2) = (2Y_1 Z_1, Y_1^2 + lnZ_1^2)$. 
Since $Y_1^2-lnZ_1^2=1$, we see that
$$
(1-Y_2,-X_2 H,l(-Y_2-1))=-2l(n Z_1^2,Z_1 Y_1 H,Y_1^2).
$$
Hence $\langle u,v \rangle=1$ for
$u=-(nZ_1^2, Z_1 Y_1 H,Y_1^2)$.
%
%
\end{itemize}

Thus, we have the following theorem. 

\begin{thm}\label{thm:general}
\begin{enumerate}
\item
Assume that $\sqrt{ln} \not \in \Q$.
Let $u$ be the Mukai vector which determines the boundary of movable cones of $\mathrm{Km}^{l-1}(A)$. 
\begin{enumerate}
\item[(1)] Assume that $l \mid Y+1$ for $Y=Y_1$ or $Y=-Y_1$.
We set $X=X_1$ or $-X_1$ according as $Y=Y_1$ or $Y=-Y_1$. 
Then,  
\begin{itemize}
\item[if] $\gcd(\frac{Y+1}{l},\frac{X}{l},Y-1)=1$, 
$\innerpro{u}{v} = 2$ for 
$u = (\frac{Y+1}{l}, \frac{X}{l}H, Y-1)$.
\item[if] $\gcd(\frac{Y+1}{l},\frac{X}{l},Y-1)=2$, 
$\innerpro{u}{v} = 1$ for 
$u = (\frac{Y+1}{2l}, \frac{X}{2l}H, \frac{Y-1}{2})$. 
\end{itemize}
\item[(2)] Let $l \nmid (Y_1 \pm 1)$. 
Then we have $\innerpro{u}{v} = 1$ for 
$u=-(nZ_1^2, Z_1 Y_1 H,Y_1^2)$.
\end{enumerate}
\item
Assume that $\sqrt{ln}  \in \Q$.
Then $\Mov(\Km^{l-1}(A))=\R_{\geq 0}h+\R_{\geq 0}(h-\sqrt{\frac{n}{l}}\delta)$.
\end{enumerate}

\end{thm}


\begin{center}

\end{center}

D{\footnotesize EPARTMENT OF} M{\footnotesize ATHEMATICS}, F{\footnotesize ACULTY OF} S{\footnotesize CIENCE}, K{\footnotesize OBE} U{\footnotesize NIVERSITY}, K{\footnotesize OBE}, 657, J{\footnotesize APAN}

{\it E-mail address:}\ akrmori@math.kobe-u.ac.jp

\end{document}